\documentclass{amsart}
\usepackage{amssymb,amsmath,latexsym}
\usepackage{amsthm}
\usepackage{fontenc}
\usepackage{amssymb}
\numberwithin{equation}{section}

\newtheorem{theorem}{Theorem}[section]

\newtheorem{lemma}[theorem]{Lemma}

\begin{document}
\author{Alexander E. Patkowski}
\title{On a Bessel function series related to the Riemann xi function}

\maketitle
\begin{abstract} We establish new Fourier integral evaluations involving the Riemann xi function related to a series involving Bessel function of the first kind. We show this infinite series involving the Bessel function of the first kind solves a boundary value problem for the cylindrical heat equation.  \end{abstract}

% AMS keywords (used in AMS journals)

% AMS subject classifications (used in AMS journals)

\section{Introduction and Main formulas}
The Riemann xi function is $\Xi(t)=\xi(\frac{1}{2}+it),$ where $\xi(s):=\frac{1}{2}s(s-1)\pi^{-\frac{s}{2}}\Gamma(\frac{s}{2})\zeta(s),$ $\Gamma(\frac{s}{2})$ being the gamma function, and the Riemann zeta function $\zeta(s)$ is defined for $\Re(s)>1,$ to be the series $\sum_{n\ge1}n^{-s}$ (see [7, 13] for background information). Evaluations of Fourier integrals for the Riemann xi function have a long history, dating back to the work of Koshliakov and Ramanujan [8, 12]. The classical application is found in Titchmarsh [13] to illustrate a proof of Hardy's theorem that infinitely many critical zeros of the Riemann zeta function lie on the line $\frac{1}{2}+it,$ $t\in\mathbb{R}.$ Perhaps the most common example is [13]
\begin{equation}\int_{0}^{\infty}\frac{\Xi(y)}{y^2+\frac{1}{4}}\cos(xy)dy=\frac{\pi}{2}\left(e^{x/2}-2e^{-x/2}\psi(e^{-2x})\right),\end{equation}
 which gives a nice connection to the theta function $\psi(y)=\sum_{n\ge1}e^{-n^2y}.$ In fact, (1.1) implies the well known functional equation [9, pg.120, eq.(4.1.12)] by replacing $x$ by $-x.$ \par
Many authors [3, 4, 5, 8, 10, 12] have developed further evaluations of integrals of the same class as (1.1), many of which contain certain series involving Bessel functions. Infinite series identities involving Bessel functions have been extended by several authors [1, 2] to have relevance in a number theoretic context. The purpose herein is to note an example which appears to have a connection to a curious boundary value problem. Theorems 1.2 and 1.3 will provide new solutions to a cylindrical heat equation with certain boundary conditions given in the following section. The proofs of our main theorems can be found in the third section. As usual [9, pg.115] (or [6]), the Bessel series of the first kind is given as
$$J_v(r)=\sum_{n=0}^{\infty}\frac{(-1)^n}{n!\Gamma(n+v+1)}\left(\frac{r}{2}\right)^{2n+v},$$ and the confluent hypergeometric function is given by [9, pg.193], $|x|<\infty,$
 $$\frac{\Gamma(a)}{\Gamma(b)}{}_1F_1(a;b;x)=\sum_{n=0}^{\infty}\frac{\Gamma(a+n)}{\Gamma(b+n)}\frac{x^n}{n!}.$$
First we offer a connection between the square of the Riemann xi function and the Bessel series presented in this study.
\begin{theorem} For $x\in\mathbb{R},$ and real number $t>0,$
$$\int_{0}^{\infty}\frac{\Xi^2(y)}{(y^2+\frac{1}{4})^2}\Big|{}_1F_1\left(\frac{1}{4}+\frac{iy}{2};1;-\frac{r^2}{4t}\right)\Big|^2\cos(xy)dy=e^{-x/2}\sqrt{t}\int_{0}^{\infty}H_1(ye^{-x})H_1(y)dy,$$
where
$$H_1(y):=\sum_{n=1}^{\infty}e^{-tn^2y^2}J_0(rny)-2\frac{\sqrt{\pi}}{\sqrt{t}y}{}_1F_1\left(\frac{1}{2};1;-\frac{r^2}{4t}\right).$$

\end{theorem}

Next we offer a curious integral formula representation for $H_1(y),$ which can be seen as a generalized form of the functional equation for the Jacobi theta function [9, pg.120, eq.(4.1.12)]. This can be seen by setting $r=0,$ and evaluating the integral on the right side. (This is also observed from applying the residue theorem to the last equation in (3.7) below.)

\begin{theorem} For $x>0,t>0,$ and $|\arg(-\frac{r^2}{t})|<\frac{\pi}{2},$
$$\sum_{n=1}^{\infty}e^{-tn^2x^2}J_0(rnx)-\frac{\sqrt{\pi}}{2\sqrt{t}x}{}_1F_1\left(\frac{1}{2};1;-\frac{r^2}{4t}\right)$$
\begin{equation}=\frac{e^{-\frac{r^2}{4t}}}{2}\int_{1}^{\infty} \frac{1}{\sqrt{y^2-1}}\left(\frac{\pi}{x\sqrt{t}}\sum_{n\ge1}G_1\left(\frac{ny\pi}{x\sqrt{t}}\right)-\frac{1}{ y}\int_{0}^{\infty}G_1(w)dw\right)dy,\end{equation}

where $G_1(z)=e^{-z^2}zJ_{0}(2\sqrt{-\frac{r^2}{4t}}z).$

\end{theorem}
To see how to obtain the functional equation for the Jacobi theta function, note from [9, pg.95, eq.(3.3.28)] that for $\Re(s)<1,$
$$\frac{2}{\sqrt{\pi}}\int_{1}^{\infty}\frac{y^{s-1}}{\sqrt{y^2-1}}dy=\frac{\Gamma(\frac{1-s}{2})}{\Gamma(\frac{1-s+1}{2})}.$$
It may also be seen from this integral coupled with Parseval's theorem for Mellin transforms that for $a>0,$
$$\int_{1}^{\infty}\frac{ye^{-ay^2}}{\sqrt{y^2-1}}dy=\frac{e^{-a}}{\sqrt{a}}.$$ Now setting $r=0$ in Theorem 1.2 gives [9, pg.120, eq.(4.1.12)] once applying these integral evaluations with $J_0(0)=1,$ by absolute convergence. \par
We mention in passing that the asymptotic estimate [9, pg.130] (or [9, pg.209, eq.(5.4.7)]) of $J_0(r)$ for large $r$ may be applied to see that 
$$\sqrt{r}\sum_{n=1}^{\infty}J_0(nr)e^{-n^2\kappa t}=\frac{1}{\sqrt{\pi}}\sum_{n=1}^{\infty}\frac{1}{\sqrt{n}}\left(\cos(nr)+\sin(nr)\right)e^{-n^2\kappa t}+O(r^{-1}),$$
as $r\rightarrow\infty.$ While the series on the right side of this estimate is not quite the Jacobi theta function, it does indeed still provide a solution to the heat equation. When $t=0$ the series only converges conditionally by the oscillatory nature of the Bessel function. For an expansion of the Fourier series on the right side when $t=0,$ see [9, pg.128, eq.(4.2.14)]. 
\par 
The next theorem shows that it is also possible to obtain a Riemann xi Fourier integral
directly connected with $H_1(y).$
\begin{theorem} For $s\in\mathbb{C},$ set $$F^{+}(x,s):=e^{xs}{}_1F_1\left(s/2;1;-\frac{r^2}{4t}\right)+e^{x(1-s)}{}_1F_1\left((1-s)/2;1;-\frac{r^2}{4t}\right).$$ Then for $x\in\mathbb{R},$

$$\int_{0}^{\infty}\frac{\Xi(y)}{(y^2+\frac{1}{4})}F^{+}\left(x,\frac{1}{2}+iy\right)dy=\sum_{n=1}^{\infty}e^{-tn^2e^{-2x}}J_0(rne^{-x})-\frac{\sqrt{\pi}}{2\sqrt{t}}{}_1F_1\left(\frac{1}{2};1;-\frac{r^2}{4t}\right)e^{x}.$$

\end{theorem}

\section{The associated Boundary value problem}  
The solution of (2.1) with the provided boundary conditions is given as our series in Theorem 1.1. We consider the heat equation in cylindrical form, where $u(r,t)$ represents temperature in a cylinder with radius $2\pi(n+1),$ for integers $n\ge1,$
\begin{equation}\kappa \nabla^2u\equiv\kappa\left(\frac{\partial^2 u}{\partial r^2}+\frac{1}{r}\frac{\partial u}{\partial r}\right)=\frac{\partial u}{\partial t},\end{equation}
with the boundary conditions:
\\*
\hspace{2mm}(i) $u(0,t)=\sqrt{\frac{\pi}{t}}u(0,\frac{\pi^2}{t})+\frac{1}{2}\sqrt{\frac{\pi}{t}}-\frac{1}{2},$ for $t>0,$
\\*
\hspace{2mm}(ii) $u(r,0)=-\frac{1}{2}+\frac{1}{r}+2\sum_{m=1}^{n}\frac{1}{\sqrt{r^2-4m^2\pi^2}},$ for $2n\pi<r<2(n+1)\pi.$
\\*
The condition (i) is the functional equation for the theta function [9, pg.120, eq.(4.1.12)], and condition (ii) is equivalent to a series evaluation found in [6, pg.936]. The Bessel series expansions which include (ii) are referred to as Schl$\ddot{o}$milch expansions [14, pg.621]. The solution is given as our series in Theorem 1.1:
\begin{equation}u(r, t)=\sum_{n=1}^{\infty}J_0(nr)e^{-n^2\kappa t}.\end{equation}
This can be shown through the standard separation of variable method [11, pg.316]. Alternatively, one can show this through the known relations
$$\frac{\partial J_0(nr)}{\partial r}=-nJ_1(nr),$$
$$\frac{\partial^2 J_0(nr)}{\partial r^2}=-n^2\frac{\partial J_1(nr)}{\partial r}=\frac{n}{r}J_1(nr)-n^2J_0(nr).$$
The growth condition $\lim_{r\rightarrow\infty}u(r,t)=0,$ may also be included in this problem, but $u(r,0)=0$ can not be satisfied for any $r>2\pi.$ The initial temperature satisfies $-1<u(r,0)<1$ since $n\ge1$ and $|J_0(x)|\le1.$ For example, if $r=3\pi$ then $n=1,$ and 
$$u(3\pi,0)=-\frac{1}{2}+\frac{1}{3\pi}+\frac{2}{\sqrt{5}\pi}<0.$$ On the other hand $u(7,0)>0.$ Consequently, since the Bessel function is oscillatory, the exponential component dampens terms so that $u(r,t)$ for $t>0$ may possibly be larger than the initial temperature. For example, it can be seen using [14, pg.684], $J_0(9.42)=-0.1803648$ which is very close to $J_0(3\pi).$ At time $t=5/\kappa$ we can approximate $u(3\pi,5/\kappa)$ to be nearly $-0.0012$ by multiplying $J_0(9.42)$ by $e^{-5}.$ Since further terms are very small due to the dampening exponential component, $u(3\pi,0)<u(3\pi,5/\kappa)<0.$ Therefore, the boundary value problem represents temperature in a cylinder with radius $2\pi(n+1),$ $n\ge1,$ which is oscillatory within $(-1,1)$ as time passes. The general solution to (2.1) with initial temperature $u(r,0)=0,$ was documented in [14, pg.616] and is attributed there to Fourier. The detailed solution by separation of variables may be found in [11, pg.317, eq.(10)], and should be compared with (2.2).

\section{Proof of Main results}
In proving our Riemann xi function integrals we will regularly make use of Parseval's theorem for Mellin transforms [9, pg.83, eq.(3.1.11)]. Define
\begin{equation}\int_{0}^{\infty}f(y)y^{s-1}dy=:\mathfrak{F}(s),\end{equation}
with inverse given by the line integral [9, pg.89],
\begin{equation}f(y)=\frac{1}{2\pi i}\int_{(c)}\mathfrak{F}(s)y^{-s}ds:=\frac{1}{2\pi i}\int_{c-i\infty}^{c+i\infty}\mathfrak{F}(s)y^{-s}ds.\end{equation}
\begin{lemma} ([9, pg.83, eq.(3.1.11)]) Suppose $\Re(s)=c$ is a real number confined to a region where $\mathfrak{F}(s)\mathfrak{G}(1-s)$ is analytic, and $\mathfrak{G}(s)$ denotes the Mellin transform of $g(y).$ Then,
$$\int_{0}^{\infty}f(y)g(y)dy=\frac{1}{2\pi i}\int_{c-i\infty}^{c+i\infty}\mathfrak{F}(s)\mathfrak{G}(1-s)ds.$$ 
\end{lemma}

\begin{proof}[Proof of Theorem 1.1] By [6, pg.706, eq.6.631, \#1] for $\Re(t)>0,$ $\Re(s+v)>0,$
\begin{equation} \int_{0}^{\infty}y^{s-1}e^{-ty^2}J_v(ry)dy=\frac{r^v\Gamma(\frac{v}{2}+\frac{s}{2})}{2t^{(s+v)/2}\Gamma(v+1)}{}_1F_1\left(\frac{v+s}{2};v+1;-\frac{r^2}{4t}\right).\end{equation}
Setting $v=0,$ in (3.3), applying Mellin inversion, and inverting the desired sum gives us
\begin{equation}\sum_{n=1}^{\infty}e^{-tn^2y^2}J_0(rny)=\frac{1}{2\pi i}\int_{(p)}\frac{\Gamma(\frac{s}{2})}{2t^{s/2}}{}_1F_1\left(\frac{s}{2};1;-\frac{r^2}{4t}\right)\zeta(s)y^{-s}ds,\end{equation}
for $p>1,$ by absolute convergence. Recall [9, pg.121, eq.(4.1.15)] $|\zeta(\sigma\pm it)|$ has polynomial growth as $t\rightarrow\infty.$ From asymptotic expansions applied similarly in [3, eq.(2.8), $\lambda=(s-1)/2,$ $z=\frac{r^2}{4t}$] it can be seen that 
$${}_1F_1\left(\frac{s}{2};1;-\frac{r^2}{4t}\right)\sim \frac{1}{\sqrt{\pi}}\left(\frac{r^2(s-1)}{8t}\right)^{-1/4}\cos\left(2\sqrt{r^2(s-1)/(8t)}-\frac{\pi}{4} \right),$$ as $|s|\rightarrow\infty,$
by Kummer's transformation [9, pg.247, $a=s/2,$ $b=1,$ $z=-\frac{r^2}{4t}$],
$${}_1F_1\left(\frac{s}{2};1;-\frac{r^2}{4t}\right)=e^{-\frac{r^2}{4t}}{}_1F_1\left(1-\frac{s}{2};1;\frac{r^2}{4t}\right).$$

Hence, the integrand in (3.4) decays exponentially as $s\rightarrow \sigma\pm i\infty,$ due to Stirling's formula for the gamma function. Therefore, we may displace the contour to the left to $\Re(s)=d,$ $0<d<1,$ thereby computing residue at the pole $s=1$ to get
\begin{equation}H_1(y)=\sum_{n=1}^{\infty}e^{-tn^2y^2}J_0(rny)-\frac{\sqrt{\pi}}{2t^{1/2}}{}_1F_1\left(\frac{1}{2};1;-\frac{r^2}{4t}\right)y^{-1}\end{equation}
$$=\frac{1}{2\pi i}\int_{(d)}\frac{\Gamma(\frac{s}{2})}{2t^{s/2}}{}_1F_1\left(\frac{s}{2};1;-\frac{r^2}{4t}\right)\zeta(s)y^{-s}ds,$$
for $0<d<1.$ Now applying Lemma 3.1 to (3.5) with $f(y)=H_1(ye^{-x}),$ $g(y)=H_1(y),$ gives

\begin{equation}\sqrt{t}\int_{0}^{\infty}H_1(ye^{-x})H_1(y)dy\end{equation}
$$=\frac{1}{2\pi i}\int_{(d)}\Gamma(\frac{s}{2})^2{}_1F_1\left(\frac{s}{2};1;-\frac{r^2}{4t}\right){}_1F_1\left(\frac{1-s}{2};1;-\frac{r^2}{4t}\right)\zeta(s)^2\pi^{\frac{1}{2}-s}e^{xs}ds,$$

where we have employed the functional equation for the Riemann xi function. Putting $d=\frac{1}{2},$ multiplying by $e^{-x/2},$ and making the change of variable $s=\frac{1}{2}+it$ now gives the result.
\end{proof}

\begin{proof}[Proof of Theorem 1.2] Recalling (3.5), we have

 \begin{equation}\begin{aligned}H_1(x)&=\sum_{n=1}^{\infty}e^{-tn^2x^2}J_0(rnx)-\frac{\sqrt{\pi}}{2t^{1/2}}{}_1F_1\left(\frac{1}{2};1;-\frac{r^2}{4t}\right)x^{-1}\\
&=\frac{1}{2\pi i}\int_{(d)}\frac{\Gamma(\frac{s}{2})}{2t^{s/2}}{}_1F_1\left(\frac{s}{2};1;-\frac{r^2}{4t}\right)\zeta(s)x^{-s}ds\\
&=\frac{1}{2\pi i}\int_{(1-d)}\frac{\Gamma(\frac{1-s}{2})}{2t^{(1-s)/2}}{}_1F_1\left(\frac{1-s}{2};1;-\frac{r^2}{4t}\right)\zeta(1-s)x^{s-1}ds\\
&=\frac{1}{2\sqrt{\pi} i}\int_{(1-d)}\frac{\Gamma(\frac{s}{2})}{2t^{(1-s)/2}}{}_1F_1\left(\frac{1-s}{2};1;-\frac{r^2}{4t}\right)\zeta(s)\pi^{-s}x^{s-1}ds ,\end{aligned}\end{equation}
for $0<d<1$ by the functional equation for $\zeta(s)$ after making a change of variable $s$ to $1-s.$ From [9, pg.95, eq.(3.3.28)], $c>0,$ $a>-1,$
\begin{equation}\frac{1}{2\pi i}\int_{(c)}x^{-s}\frac{\Gamma(s)}{\Gamma(s+a+1)}ds=\frac{(1-x)^{a}}{\Gamma(a+1)},\end{equation}
if $0<x\le1,$ and $0$ if $x>1.$ 
From [6, pg.1023, eq.9.211,\#3], we have that for $\Re(\alpha+v+1)>0,$ $|\arg(x)|<\frac{\pi}{2},$
\begin{equation}{}_1F_{1}\left(-v,\alpha+1;x\right)=\frac{\Gamma(\alpha+1)e^{x}x^{-\alpha/2}}{\Gamma(\alpha+v+1)}\int_{0}^{\infty}e^{-y}y^{v+\alpha/2}J_{\alpha}(2\sqrt{xy})dy,\end{equation}
Putting $v=(s-1)/2,$ and $\alpha=0,$ in (3.9) gives for $\Re(s+1)>0,$ $|\arg(x)|<\frac{\pi}{2},$
\begin{equation}{}_1F_{1}\left(\frac{1-s}{2},1;x\right)=\frac{e^{x}}{\Gamma(\frac{s+1}{2})}\int_{0}^{\infty}e^{-y}y^{(s-1)/2}J_{0}(2\sqrt{xy})dy,\end{equation}
 
By the formula [13, pg.29, eq.(2.11.1)], if $F(x)$ satifies (i) $F'(x)$ is bounded in any finite interval, (ii) continuous, and (iii) decays like $O(x^{-N}),$ $N>1$ as $x\rightarrow\infty,$ then 
\begin{equation} \zeta(s)\int_{0}^{\infty}y^{s-1}F(y)dy=\int_{0}^{\infty}y^{s-1}\left(\sum_{n\ge1}F(ny)-\frac{1}{y}\int_{0}^{\infty}F(w)dw\right)dy,\end{equation} for $0<\Re(s)<1.$
Notice that the common region of holomorphy of (3.8) with $s$ replaced by $1-s,$ and (3.10) is $\{s: -1<\Re(s)<1\}.$ Hence, putting $a=-\frac{1}{2}$ in (3.8) with $s\rightarrow s/2,$ and applying (3.10) to Lemma 3.1 with (3.11), the last line of (3.7) gets transformed into
\begin{equation}\begin{aligned}&\frac{1}{2\sqrt{\pi} i}\int_{(1-d)}\frac{\Gamma(\frac{s}{2})e^{-\frac{r^2}{4t}}}{2\Gamma(\frac{s+1}{2})t^{(1-s)/2}}\zeta(s)\left(\int_{0}^{\infty}e^{-z}z^{(s-1)/2}J_{0}(2\sqrt{-\frac{r^2}{4t}z})dz\right)\pi^{-s}x^{s-1}ds\\
&=\frac{1}{2\sqrt{\pi} i}\int_{(1-d)}\frac{\Gamma(\frac{s}{2})e^{-\frac{r^2}{4t}}}{2\Gamma(\frac{s+1}{2})t^{(1-s)/2}}\left(\int_{0}^{\infty}z^{s-1}\left(\sum_{n\ge1}G_1(nz)-\frac{1}{y}\int_{0}^{\infty}G_1(w)dw\right)dz\right)\\
&\times\pi^{-s}x^{s-1}ds\\
&=\frac{e^{-\frac{r^2}{4t}}}{2\sqrt{t}}\int_{1}^{\infty} \frac{\pi}{\sqrt{y^2-1}}\left(\frac{1}{x}\sum_{n\ge1}G_1\left(\frac{ny\pi}{x\sqrt{t}}\right)-\frac{x\sqrt{t}}{\pi yx}\int_{0}^{\infty}G_1(w)dw\right)dy\\
&=\frac{e^{-\frac{r^2}{4t}}}{2}\int_{1}^{\infty} \frac{1}{\sqrt{y^2-1}}\left(\frac{\pi}{x\sqrt{t}}\sum_{n\ge1}G_1\left(\frac{ny\pi}{x\sqrt{t}}\right)-\frac{1}{ y}\int_{0}^{\infty}G_1(w)dw\right)dy, \end{aligned}\end{equation}

where $G_1(z)=e^{-z^2}zJ_{0}(2\sqrt{-\frac{r^2}{4t}}z).$
\end{proof}

\begin{proof}[Proof of Theorem 1.3] This result follows from the second line in (3.7) with $x$ replaced by $e^{-x}$ upon noting $F^{+}\left(x,\frac{1}{2}+it\right)=F^{+}\left(x,\frac{1}{2}-it\right).$

\end{proof}

1390 Bumps River Rd. \\*
Centerville, MA
02632 \\*
USA \\*
ul. A. E. Ody\'{n}ca 47 \\*
02-606 Warsaw\\*
Poland\\*
E-mail: alexpatk@hotmail.com, alexepatkowski@gmail.com

\end{document}